\documentclass[a4paper,reqno]{amsart}

\usepackage[utf8]{inputenc}
\usepackage[UKenglish]{babel}
\usepackage{mathtools}

\usepackage{amsmath, amsthm, amsfonts, amssymb, xcolor}
\usepackage{mathrsfs}
\usepackage{dsfont}
\usepackage[shortlabels]{enumitem}

\usepackage{hyperref}

\theoremstyle{plain}
\newtheorem{theorem}{Theorem}[section]

\newtheorem{lemma}[theorem]{Lemma}
\newtheorem{prop}[theorem]{Proposition}

\theoremstyle{definition}
\newtheorem{remark}[theorem]{Remark}

\newcommand{\cball}{\overline{\mathds{B}}}
\newcommand{\real}{\mathds{R}}
\newcommand{\Pp}{\mathds{P}}
\newcommand{\Ee}{\mathds{E}}
\newcommand{\I}{\mathds{1}}
\newcommand{\Dcal}{\mathcal{D}}
\newcommand{\Fcal}{\mathcal{F}}

\newcommand{\tinybullet}{\scriptscriptstyle\bullet}

\begin{document}

\title[Time changed multidimensional Wiener process]{Functional limit theorems for a time-changed multidimensional  Wiener process}

\author{Yuliya Mishura}
\address{Department of Probability, Statistics and Actuarial Mathematics, Taras Shevchenko National University of Kyiv, 01601  Kyiv, Ukraine}
\email{yuliyamishura@knu.ua}

\author{Ren\'e L.\ Schilling}
\address{Institut f\"{u}r Mathematische Stochastik, Fakult\"{a}t Mathematik, TU Dresden, 01062 Dresden, Germany}
\email{rene.schilling@tu-dresden.de}

\begin{abstract}
     We study the asymptotic behaviour of a properly normalized time-changed multidimensional Wiener process; the time change is given by an additive functional of the Wiener process itself. At the level of generators, the time change means that we consider the Laplace operator -- which generates a multidimensional Wiener process -- and multiply it by a (possibly degenerate) state-space dependent intensity. We assume that the intensity admits limits at infinity in each octant of the state space, but the values of these limits may be different. Applying a functional limit theorem for the superposition of stochastic processes, we prove functional limit theorems for the normalized time-changed multidimensional Wiener process. Among the possible limits there is a multidimensional analogue of skew Brownian motion.
\end{abstract}

\dedicatory{To the memory of Yuri Kondratiev}

\keywords{Multidimensional Wiener process; time change; functional limit theorem; multidimensional skew Brownian motion}

\subjclass[2020]{\emph{Primary:} 60J65. \emph{Secondary:} 60J60; 60J55; 60F05; 60F17.}

\thanks{The authors thank Michael R\"{o}ckner for fruitful discussions on this topic. The first-named author was supported by the Swedish Foundation for Strategic Research (Grant No. UKR24-0004), the Japan Science and Technology Agency CREST project (Grant No. 811JPMJCR2115) and the ToppForsk project of the Research Council of Norway with the title STORM: Stochastics for Time-Space Risk Models (Grant No.\ 274410). The second-named author was supported by the 6G-life and ScaDS.AI Dresden/Leipzig projects.}

\maketitle

\section{Introduction}\label{intro}
In this paper we study the asymptotic behaviour of a normalized and time-changed multidimensional Wiener process. Our motivation comes from the study of solutions to the following type of parabolic Cauchy problems
\begin{gather}\begin{alignedat}{3}\label{cauchy}
    \frac{\partial}{\partial t} u(t,x) &= \lambda(x)\Delta u(t,x), &\quad& t\geq 0,\; x\in \real^d,\\
                                u(0,x) &= f(x), &\quad& t=0,\; x\in \real^d,
\end{alignedat}\end{gather}
where the coefficient $\lambda$ can be irregular and even degenerate.

Following the suggestions of Yuri Kondratiev, we studied in the joint paper~\cite{KMS} the Cauchy problem \eqref{cauchy} in dimension one, $d=1$. Here we are going to generalize the problem \eqref{cauchy} to higher dimensions. Quite often it is assumed in such problems that the diffusion coefficient is uniformly elliptic. In the situation considered here this reduces to the assumption that  $\lambda(x)$ is bounded away from zero; degenerate problems are often treated with the method of vanishing viscosity, see e.g.\ Ole\u{\i}nik \& Radkevi\v{c}~\cite[\S~III.2]{ole-rad} and Bogachev \emph{et al.}~\cite[\S 6.7.(ii)]{bog-et-al}, and the literature mentioned in these references.

An alternative method to study such problems is to use a probabilistic approach. Assume that we can construct a Markov process $X$, starting from any point $x\in\real^d$, and with infinitesimal generator $\lambda(x)\Delta$. In this situation, we have the following stochastic representation of the parabolic Cauchy problem \eqref{cauchy}:
\begin{gather*}
    u(t,x)= \Ee^x [f(X_t)].
\end{gather*}
Yet, a direct construction of this process may be difficult (if it is possible at all), if one wants to use stochastic differential equations (SDEs) in the case when the diffusion coefficient $\sqrt{\lambda(x)}$ is irregular and degenerate. Typically, diffusions in random media lead to random coefficients $\lambda =\lambda(\omega,x)$, which have an additive form with respect to the points in the medium. This structure corresponds to the energy of an diffusing particle in the medium.  Such coefficients cannot be uniformly positive, in general. A standard reference for stochastic differential equations with singular coefficients is Cherny \& Engelbert~\cite{cher}.

A possible way out is offered by the theory of Dirichlet forms, see e.g.\ Fukushima \emph{et al.}~\cite{fot}, but the resulting process may be defined only up to an exceptional set. Just as in the one-dimensional case, see~\cite{KMS}, we will use a different approach, namely additive functionals and random time changes. Recall the following statement which is proved differently and for different purposes in Böttcher \emph{et al.}~\cite{BSW} and Kurtz~\cite{kurtz}.
\begin{prop}\label{prop1}
    Let $\left\{X_t,\, t\geq 0\right\}$ be a $d$-dimensional Feller process with natural filtration $\Fcal^X$, $\Fcal_t^X = \sigma(X_s, s\leq t)$, and generator $(A, \Dcal(A))$. Assume that the function $\lambda\in C_b(\real^d)$ is real-valued and strictly positive. Set
    \begin{gather*}
        S_X(t,\omega) := \int_0^t\frac{ds}{\lambda(X_s(\omega))}
    \end{gather*}
    and denote by $\tau_t(\omega) = \inf\{u>0: S_X(u,\omega)>t\}$ $(\inf\emptyset:=+\infty)$ the generalized right-continuous inverse. Then the time-changed process
    $
        \left\{X_{\tau_t},\, t\geq 0\right\}
    $
    with filtration
    $
        \left\{\Fcal_{\tau_t}^X,\,t\geq 0\right\}
    $
    is again a Feller process, and its generator is the closure of $(\lambda(\cdot)A, \Dcal(A))$.
\end{prop}
Due to the different recurrence and transience behaviour of one- and multi-dimensional diffusion processes, we restricted ourselves in the paper~\cite{KMS} to the case $d=1$; we will now consider the case where $X_t=B_t$ is a multidimensional Wiener process. Throughout, we assume that $\lambda:\real^d\to [0,\infty)$ is a positive measurable function, which is Lebesgue a.e.\ strictly positive, i.e.\ $\mathrm{Leb}\{x:\lambda(x)=0\}=0$. In our main result, Theorem~\ref{theor1-1}, we do not require that $\lambda$ is bounded or bounded away from zero. However, in fact  we assume that $\lambda(x)$ is bounded away from zero for $|x|\gg 1$, since the limit exists and is strictly positive, see assumption  \ref{vi}. Since $B_t$ has a transition density w.r.t.\ Lebesgue measure, the a.e.\ strict positivity of $\lambda$ guarantees that $S_B(t,\omega)$ is a.s.\ strictly increasing and continuous; thus, $\tau_t$ is also strictly increasing and continuous. Finally, we assume that $1/\lambda$ is locally integrable, but not necessarily of class $L^1(\real^d,dx)$. Note that under these assumptions, the time-changed process $B_{\tau_t}$ is both a diffusion process and a local martingale w.r.t.\ the filtration $\Fcal_{\tau_t}$; in fact, writing $Y_t = B_{\tau_t}$, we have
\begin{gather}\label{mart-repres}
    Y_t=Y_0+\int_0^t \sqrt{\lambda(Y_s)}\,d\tilde{B}_s,
\end{gather}
with some $d$-dimensional Wiener process $\tilde{B}$.

Our aim is to establish functional limit theorems for the process $B_{\tau_t}$ without resorting to the martingale representation \eqref{mart-repres} or the Feller property and the form of its generator; instead, we want to focus on the pathwise representation as superposition of two stochastic processes. Using this approach, we need neither that the time change $\tau$ and the original process $X$ are independent (e.g.\ as in Bochner's subordination) nor that $\lambda$ is regular (as in the Feller case) or non-degenerate (as in the SDE approach). There exists a substantial literature on random-time changes, using various perspectives. Let us mention a few of them: The monograph by Fukushima \emph{et al.}~\cite{fot} treats time changes of symmetric Markov processes, and Harlamov~\cite{harl} describes the representation  of semi-Markov processes in the form of a Markov process, transformed by a time change. The paper~\cite{Magdziarz} discusses the asymptotic behaviour of the standard one-dimensional Brownian motion time changed by the (generalized) inverse of an (independent) subordinator. A detailed bibliography on the properties of time-changed processes can be found in the paper~\cite{amp}.

Our paper is organized as follows. Since it is important to have a properly defined time change for any $t>0$, we study in Section~\ref{sec-2} the asymptotic behaviour of some additive functionals of integral type for a multi-dimensional Brownian motion, and we introduce conditions that ensure the existence of the time change $\tau_t$ for any $t>0$. Section~\ref{sec-3} contains the main results, namely functional limit theorems for a time-changed multidimensional Wiener process. In Section~\ref{sub1} we briefly recall (a modification of) a theorem on weak convergence of superpositions of stochastic processes from Silvestrov~\cite{Silvestrov}. This is then applied to our setting, first (in Section~\ref{sub2}) with the simplifying assumption that $\lambda$ is separated from zero. Then we remove this restriction, assuming instead that $\lambda$ has ``radial'' limits, or alternatively, limits in the   ``diagonal'' direction   of any octant.

\section{Asymptotic behaviour of integral functionals of\\ multi-dimensional Brownian motion}\label{sec-2}

Consider a $d$-dimensional Wiener process  $B=\{ B_t, \, t\geq0 \}$, starting from $0$, and denote by $\Fcal^B = \left\{\Fcal^{B}_{t}, \, t\geq 0 \right\}$ its natural filtration. Throughout this paper, $\lambda:\real^d\to [0,\infty)$ is a  measurable function such that $f:\real^d\to (0,\infty]$, $f(x) := 1/\lambda(x)$ is locally integrable. Under these assumptions, the following additive functional of a $d$-dimensional Wiener process is well-defined:
\begin{equation}\label{func1}
	S_B(t)
	= \int_{0}^{t} \frac{ds}{\lambda(B_s)}
	=\int_{0}^t f(B_s)\,ds.
\end{equation}
Let $e_1,\dots,e_d$ be the canonical basis of $\real^d$, $\|x\|=\left(\sum_{i=1}^dx_i^2\right)^{1/2}$ the standard Euclidean norm, and  denote by $\Delta_\alpha := \left\{ \sum_{i=1}^d(-1)^{\alpha_i}t_i e_i \::\: t_1,\dots,t_d > 0\right\}$ the open ``octants'' of $\real^d$ where $\alpha \in \{0,1\}^d$ is a multiindex. For example, in $\real^2$ the $4$ quadrants $Q_I,Q_{II},Q_{III},Q_{IV}$ correspond to $\Delta_{(0,0)}, \Delta_{(1,0)}, \Delta_{(1,1)}, \Delta_{(0,1)}$. Clearly, $\Delta_0 = \real_+^d$ with $\real_+ = (0,\infty)$. We exclude the bounding hyperplanes of the octants from our considerations, since they will not lead to any difficulties due to the smoothness of the distribution of the multidimensional Wiener process.

We are interested in conditions that ensure $\int_{0}^{\infty} \lambda(B_s)^{-1}\,ds =\infty$ a.s.  In other words, the first issue is to study the asymptotic behaviour of $\int_{0}^{t} \lambda(B_s)^{-1}\,ds$ as $t \to \infty$. In the one-dimensional case the situation is very simple: due to the recurrence of $B$, the local time at the point $a\in\real$ satisfies
\begin{equation*}
	L_{\infty} (a) = \infty \quad \text{a.s.\ for almost all} \quad a \in \real.
\end{equation*}
Therefore, by the occupation time formula,
\begin{equation*}
	\int_{0}^{\infty} \frac{ds}{\lambda(B_s)}
	= \int_{\mathds{R}} L_{\infty} (x) \frac{dx}{\lambda(x)}
	= \infty \quad \text{a.s.},
\end{equation*}
see also~\cite{KMS}.

In the case $d=2$ we apply the Kallianpur--Robbins law (see \cite{Kallianpur}),  which states that for any nonnegative $g \in   L_1(\real^2)$, $g\not\equiv 0$,
\begin{equation*}
	\Pp \left\{ \frac{2\pi}{\|g\|_{  L_1(\mathds{R}^2)} \log T} \int_0^T g(B_t) dt \leq u \right\}
	\xrightarrow[T\to\infty]{} 1-e^{-u}, u>0.
\end{equation*}
In particular, if $f=\lambda^{-1}\in   L_1(\mathds{R}^2)$ is strictly positive,
\begin{equation}\label{proba}
	\int_0^T f(B_t) dt  \to +\infty\quad\text{as}\quad T \to +\infty,
\end{equation}
in probability as well as almost surely, since the integral increases. If $1/\lambda \notin    L_1(\mathds{R}^2)$, we can choose $0 \leq f_1(x) \leq f(x)= {1}/{\lambda(x)}$, $x\in \real^2$, such that $f_1 \in   L_1(\mathds{R}^2)$. Clearly, $f_1$ satisfies \eqref{proba}, and consequently
\begin{equation*}
	\int_0^T \frac{dt}{\lambda(B_t)}  \to +\infty \quad\text{as}\quad T \to +\infty.
\end{equation*}
The case $d \geq 3$ is considered in Lemma~\ref{infin1}.

For now, we discuss a weaker result which, however, sheds some light on the situation, in which we are interested.
\begin{lemma}\label{infin}
	Let $d\geq 2$ and $\lambda:  \real^d \to (0, +\infty)$ be a measurable function, satisfying the following assumptions:
\begin{enumerate}
\item\label{i} The function  $1/\lambda$ is locally integrable;
\item\label{ii} There exist $2^d$ numbers $\{\lambda_\alpha, \alpha\in\{0,1\}^d\}$ such that
\begin{equation*} 
	\lim_{\substack{\min_{1\leq i\leq d}|x_i| \to \infty\\ x \in \Delta_\alpha}} \lambda(x)
	= \lambda_\alpha > 0, \quad \alpha\in\{0,1\}^d.
\end{equation*}
\end{enumerate}
Then
\begin{gather}\label{local1}
    \Ee (S_B(\infty))
    =\Ee \left(\int_{0}^{\infty} \frac{ds}{\lambda(B_s)}\right)
      = \infty.
   \end{gather}
\end{lemma}
\begin{proof}
Obviously, we have for any $t>0$
\begin{align*}
   \Ee \left(\int_0^t \frac{ds}{\lambda(B_s)}\right)
   &= (2\pi)^{-d/2}\int_{\real^d} \frac{1}{\lambda(x)} \int_0^t s^{-d/2} e^{-\frac{\|x\|^2}{2s}}\,ds\,dx
\intertext{and the change of variable $z=\frac{\|x\|^2}{2s}$, hence $s= \frac{\|x\|^2}{2z}$ and $ds= -\frac{\|x\|^2}{2z^2}\,dz$, gives}
  \Ee \left(\int_0^t \frac{ds}{\lambda(B_s)}\right)
  &= (2\pi)^{-d/2} \int_{\real^d} \frac{1}{\lambda(x)} \int_{\frac{\|x\|^2}{2t}}^{\infty} e^{-z} \frac{(2z)^{d/2}}{\|x\|^d} \frac{\|x\|^2}{2z^2} \,dz \,dx \\
  &= \pi^{-d/2} \int_{\real^d} \frac{\|x\|^{2-d}}{\lambda(x)} \int_{\frac{\|x\|^2}{2t}}^{\infty} e^{-z} z^{d/2-2} \,dz\,dx.
\end{align*}
We distinguish between two cases:

\medskip\noindent
\textbf{Case 1:} $d=2$. Then, for any $x \in \real^d$
\begin{equation*}
    I_t(x) := \int_{\frac{\|x\|^2}{2t}}^{\infty} e^{-z} z^{-1} \,dz \uparrow \infty \quad \text{as}\quad  t \to \infty.
\end{equation*}
Therefore, by monotone convergence,
\begin{equation*}
	\int_{\real^d} \frac{\|x\|^{2-d}}{\lambda(x)} I_t(x)\,dx
	= \int_{\real^2} \frac{1}{\lambda(x)} I_t(x)\,dx\xrightarrow[t\to\infty]{} \infty,
\end{equation*}
and \eqref{local1} follows.

\medskip\noindent
\textbf{Case 2:} $d>2$. Then we have
\begin{equation*}
	\int_{\frac{\|x\|^2}{2t}}^{\infty} e^{-z} z^{d/2-2}\,dz
	\xrightarrow[t\to\infty]{} \Gamma \left(\frac{d}{2} -1\right).
\end{equation*}
By assumption we have in the first octant $\Delta_0 = \real_+^d$ that
\begin{equation*}
	\frac{1}{\lambda_0} = \lim_{\substack{\min_{1\leq i\leq d}|x_i| \to \infty\\ x \in \Delta_0}} \frac{1}{\lambda(x)}.
\end{equation*}
Thus, for any $A>0$
\begin{align*}
    &\int_{\real^d} \frac{\|x\|^{2-d}}{\lambda(x)}  \int_{\frac{\|x\|^2}{2t}}^{\infty} e^{-z} z^{d/2-2} \,dz \,dx \\
    &\quad\geq  \int_{[A,\, 2A]^d} \frac{ \|x\|^{2-d}}{\lambda(x)} \int_{\frac{\|x\|^2}{2t}}^{\infty} e^{-z} z^{d/2-2} \,dz \,dx\\
    &\quad\geq \int_{\frac{2A^2d}{t}}^{\infty} e^{-z} z^{d/2-2} dz \int_{[A,\, 2A]^d}  \frac{\|x\|^{2-d}}{\lambda(x)} \,dx.
\end{align*}
Let $A_0>0$ be such that for all $x\in\real^d_+$ with $\min_{1\leq i\leq d} x_i >A_0$ we have $\frac{1}{\lambda(x)} > \frac{1}{2\lambda_0}$.(Recall that $\lambda_0>0$).   Then we have for any $A\geq A_0$
\begin{align*}
    \int_{[A,\, 2A]^d} \frac{\|x\|^{2-d}}{\lambda(x)} \,dx
    &\geq \frac{1}{2\lambda_0} \int_{[1,2]^d} A^{2-d} \|y\|^{2-d}A^d \,dy\\
    &\geq \frac{A^2}{2\lambda_0}  \int_{[1,2]^d} \|y\|^{2-d} \,dy.
\end{align*}
and
\begin{align*}
	\lim_{t\uparrow \infty} \Ee \left(\int_0^t \frac{ds}{\lambda(B_s)}\right)
    &\geq \lim_{t\uparrow \infty} \int_{\frac{2A^2d}{t}}^{\infty} e^{-z} z^{d/2-2} \,dz
    \cdot \frac{A^2}{2\lambda_0}  \int_{[1,2]^d} \|y\|^{2-d} \,dy \\
    &= \frac{\Gamma\left(\frac{d}{2} -1\right)  A^2}{2\lambda_0} \int_{[1,2]^d} \|y\|^{2-d} \,dy.
\end{align*}
Letting $A \to \infty$ gives \eqref{local1}, and  the proof is complete.
\end{proof}

Let us return to the integral $S_B(t)$ itself and consider assumptions ensuring for $d>2$
\begin{gather*}
	S_B(\infty)
	= \int_{0}^{\infty} \frac{ds}{\lambda(B_s)}
  	= \infty\quad \text{a.s.}
\end{gather*}

\begin{lemma}\label{infin1}
	Let $d>2$ and $\lambda:  \real^d \to (0, +\infty)$ be a measurable function, satisfying assumption \ref{i} from  Lemma~\ref{infin} and the following assumption
\begin{enumerate}\setcounter{enumi}{2}
	\item\label{iii} $\limsup_{\|x\| \to \infty} \lambda (x)<\infty$.
\end{enumerate}
Then
\begin{gather}\label{local}
   S_B(\infty)
    = \int_{0}^{\infty} \frac{ds}{\lambda(B_s)}
    = \infty
    \quad \text{a.s.}
\end{gather}
\end{lemma}
\begin{proof}
	According to Theorem~6 from the classical paper Dvoretzky \& Erd\"{o}s~\cite{dvorerd}, for any decreasing function $g:\real_{+} \to \real_{+} $ the probability of the event
\begin{gather*}
	\left\{ \|B_t\| < g(t) \sqrt{t} \text{\ for some arbitrarily large\ }  t \right\}
\end{gather*}
is $0$ or $1$ depending on whether the series $\sum_{m=1}^{\infty} g^{d-2}(2^m)$ converges or diverges. For $g(t) = t^{-1/d}$, the series $\sum_{m=1}^{\infty} (2^m)^{-1+ \frac{2}{d}}$ converges, and we have $\|B_t(\omega)\| > t^{1/2 -1/d}$ for $t \geq t_0(\omega)$ and almost all $\omega$. Furthermore, it follows from \ref{iii} that there exist $a>0$ and $\rho >0$ such that $\frac{1}{\lambda(x)} > a$ for $\|x\| \geq \rho$. Therefore, for all $t \geq s_0(\omega) := \max\left\{t_0(\omega),\; \rho^{\frac{1}{1/2-1/d}}\right\}$
\begin{equation*}
    \int_{0}^{t} \frac{ds}{\lambda(B_s(\omega))}
    \geq \int_{s_0(\omega)}^{t} \frac{ds}{\lambda(B_s(\omega))}
    \geq a(t-s_0(\omega)) \xrightarrow[t\to\infty]{} \infty \quad \textit{a.s.},
\end{equation*}
from which \eqref{local} follows.
\end{proof}

\begin{remark}
In order to illustrate Lemma~\ref{infin1}, we consider some examples. They work also in the case $d=2$, but the Kallianpur--Robbins law provides the general result, see above. To start with, assume that $\lambda$ is a bounded function. Then
\begin{equation*}
	\int_{0}^{\infty} \frac{ds}{\lambda(B_s)} = \infty \quad \text{a.s.\ for any} \quad d \geq 1. 
\end{equation*}
	
Assume that $\lambda$  is unbounded, growing at most polynomially, i.e. 
\begin{equation*}
	\limsup_{\|x\| \to \infty} \frac{\lambda(x)}{1+\|x\|^{\gamma}} = b > 0,
\end{equation*}
for some  $\gamma >0$. Obviously, there exists some $C>0$ such that $\lambda(x) \leq C (1+\|x\|^{\gamma})$ and
\begin{gather*}
	\int_{0}^{\infty} \frac{ds}{\lambda(B_s)} \geq \int_{0}^{\infty} \frac{ds}{C(1+\|B_s\|^{\gamma})}.
\end{gather*}
According to the law of the iterated logarithm for the norm of a multidimensional Wiener process,
\begin{equation*}
	\limsup_{u \to \infty} \frac{\|B_u\|}{\sqrt{u\log\log u }} = \sqrt{2}.
\end{equation*}
This shows that a.s.\
\begin{gather*}
	\int_{0}^{\infty} \frac{ds}{\lambda(B_s)}
	\geq \int_{0}^{\infty} \frac{ds}{C\left(1+\|B_s\|^{\gamma}\right)}
	\geq C'\int_{0}^{\infty} \frac{ds}{C\left(1+s^{\frac{\gamma}{2}} (\log\log s )^{\frac{\gamma}{2}}\right)},
\end{gather*}
but the latter integral diverges if $0 \leq \gamma <2$.

We remark, in passing, that in the case $d=2$ it was mentioned by Spitzer (see \cite{spitz}) that P.~L\'evy's conjecture on planar Brownian motion: ``\emph{For every $\epsilon>0$, the lower bound $\|B_t\|\geq t^{-\epsilon}$ holds for all sufficiently large $t$ with probability $1$}'', is false. Therefore in this case (as is the case in dimension $d=1$) the asymptotic behaviour of $\lambda$ does, in general, not help.
\end{remark}

Taking into account the previous discussion, we will from now on assume that
\begin{equation*}
	\int_{0}^{\infty} \frac{ds}{\lambda(B_s)} = +\infty \quad \text{a.s.}
\end{equation*}
By $\tau_t  = \inf\left\{ x>0 : \, S_B(x) > t \right\}$ we denote the generalized right-inverse of the additive functional $S_B$ defined by  \eqref{func1}. As usual, $\inf\emptyset:=\infty$; note that $\tau_0 = 0$. Since $S_B(x)$ is continuous and strictly increasing, we have
\begin{align*}
	\{\tau_t \leq u \}
	&= \left\{S_B(u)=\int_{0}^{u} \frac{ds}{\lambda(B_s)} \geq t \right\} \in \Fcal_{u}^{B}.
\end{align*}
Therefore, $\tau_t$ is a stopping time w.r.t.\ the filtration $\Fcal^B$; moreover,  $\tau_t < \infty$ a.s., and we can consider the time-changed process $B_{\tau_t}$.  We are going to apply to this process functional limit theorems for the superpositions of stochastic  processes. Let us point out that, in general, $B$ and $\tau$ are not independent, nor is $B_{\tau_t}$ a   square-integrable martingale. In view of Wald's identities, a necessary and sufficient condition for the latter is $\Ee\left[\tau_t\right]<\infty$. The following lemma gives a typical sufficient condition for $\Ee\left[\tau_t\right]<\infty$.

\begin{lemma}
	Assume that there exists a \textup{(}non-random\textup{)} positive and strictly increasing function $\phi:\real^2 \to [0,\infty)$ such that $\lim_{t\to\infty}\phi (t)=\infty$ and
	\begin{gather*}
		\int_{0}^{t} \frac{ds}{\lambda(B_s)}\geq \phi(t),\quad t\geq 0.
	\end{gather*}
	Then $\Ee\left[\tau_t\right]<\infty$ for all $t>0$. In particular, if $\lambda$ is bounded, i.e.\ it has a uniform upper bound $\lambda(x) \leq C$ for all $x \in \real^d$, then $\Ee\left[\tau_t\right]<\infty$ for all $t>0$.
\end{lemma}
\begin{proof}
The claim follows easily from the following equalities
\begin{align*}
    \Ee\left[\tau_t\right]
    = \int_{0}^{\infty} \Pp \left(\tau_t \geq x \right) dx
    &= \int_{0}^{\infty} \Pp \left(\int_{0}^{x} \frac{ds}{\lambda(B_s)}  \leq t \right) dx \\
    &= \int_{0}^{\phi^{-1}(t)} \Pp \left(\int_{0}^{x} \frac{ds}{\lambda(B_s)}   \leq t\right) dx
    \leq \phi^{-1}(t).
\qedhere
\end{align*}
\end{proof}

\section{Functional limit theorems for a time-changed multidimensional Wiener process}\label{sec-3}
Since the map $t\mapsto\tau_t$ is continuous, we conclude that the processes $t\mapsto B_{\tau_t}$ and $t\mapsto B_{\tau_{nt}},\: n\geq 1$ are continuous, but we do not have much further information. This means, in particular, that we have to use methods which do not rely on ($L^2$\nobreakdash-)\/martingale methods or independence of the time change and the underlying Brownian motion. Our aim is to identify the normalizing factor $\phi(n)\to\infty$ as $n\to\infty$ such that, for a suitable limiting process $\zeta=\{\zeta_t, t\geq 0\}$,
\begin{gather*}
    \left\{\frac{B_{\tau_{nt}}}{\phi(n)},\, t\in[0,T]\right\} \xRightarrow[n\to\infty]{\phantom{\mathrm{fdd}}} \big\{\zeta_t, t\in[0,T]\big\}
\end{gather*}
in the sense of weak convergence of the measures corresponding to the stochastic processes on any interval $[0,T]$, $T>0$.
As usual, the symbol ``$\xRightarrow[]{\phantom{\mathrm{fdd}}}$'' denotes weak convergence of probability measures.

\subsection{A Functional limit theorem for the superposition of stochastic processes}\label{sub1}

We will use the following modification of a theorem on weak convergence of superposition of stochastic processes, see Theorem 3.2.1 and its generalization in Section~3.2 of the monograph by Silvestrov~\cite{Silvestrov}. By ``$\xrightarrow[]{\mathrm{fdd}}$'' we denote the weak convergence of the finite-dimensional distributions (fdd), whereas ``$\xRightarrow[]{\phantom{\mathrm{fdd}}}$'' denotes weak convergence of probability measures, see above.

\begin{theorem}\label{th-cond-1}
    Let $\left\{(Z_{n}(t), \nu_{n}(t)), t \geq 0\right\}_{n\geq 0}$ be a sequence of stochastic processes with continuous paths such that for all $n\geq 1$ the processes $\nu_n(t)\geq 0$ are positive and increasing in $t$,   and $Z_n(t)$ are $\real^d$-valued for some $d\geq 1$.  Assume that the following conditions hold for all $T>0$.
    \begin{enumerate}\setcounter{enumi}{3}
    \item\label{iv}
        There is some dense subset $U\subset [0,T]$ such that $0\in U$ and
        \begin{gather*}
        	\big\{ (\nu_{n}(t), Z_{n}(t)),\, t\in U\big\}
        	\xrightarrow[n\to\infty]{\mathrm{fdd}} \big\{ (\nu_{0}(t) , Z_ {0}(t)), \,  t  \in U\big\}.
        \end{gather*}
   \item\label{v}
        For any $\epsilon>0$
        \begin{gather*}
        	\lim_{h \downarrow 0} \limsup_{n\to \infty} \Pp\left(\sup_{0 \leq t_1 \leq t_2 \leq (t_1+h) \wedge T} \|Z_n (t_2) - Z_n (t_1)\|> \epsilon \right) = 0.
        \end{gather*}
    \end{enumerate}
    Then, on every interval $[0,T]$, the superposition $\zeta_n (\cdot) := Z_n (\nu_n (\cdot))$ converges weakly:
    \begin{gather*}
        \zeta_n (\cdot) \xRightarrow[n\to\infty]{} \zeta_0 (\cdot) = Z_0 (\nu_0 (\cdot)).
    \end{gather*}
 \end{theorem}

\subsection{A functional limit theorem if the intensity $\lambda$ is separated from zero}\label{sub2}

We are going to apply Theorem~\ref{th-cond-1} with $\phi(n) := \sqrt{n}$, $Z_n(t) := \frac{1}{\sqrt{n}}B_{nt}$ and $\nu_n(t) :=\frac{1}{{n}}\tau_{nt}$. In order to demonstrate the essence of our method, we assume that the function $\lambda(s)$ is uniformly bounded away from zero and admits limits in all octants $\Delta_\alpha$.

\begin{theorem}\label{separated}
Let $\lambda:  \real^d \to (0, +\infty)$ be a measurable function, satisfying the assumptions \ref{i} and \ref{ii} from
Lemma~\ref{infin}. If, additionally, $\lambda (x) \geq c$ for some $c>0$ and all $x \in \real^d $, then on any interval $[0, T]$
\begin{equation*}
	\left\{\frac{B_{\tau_{nt}}}{\sqrt{n}}, t \in [0, T] \right\}
	\xRightarrow[n\to\infty]{} \left\{ W(\nu^{-1} (t)), t \in [0, T] \right\},
\end{equation*}
where $W$ is a $d$-dimensional Wiener process and $\nu(t) = \int_0^t \upsilon (W_s)\,ds$,
\begin{gather*}
	\upsilon(x) = \sum_{\alpha\in\{0,1\}^d} \frac{1}{\lambda_\alpha} \I_{\Delta_\alpha}(x).
\end{gather*}
\end{theorem}
\begin{remark}
	The Process $W(\nu^{-1} (\cdot))$ can be considered as a multi-dimensional analogue of a skew Brownian motion.
\end{remark}

\begin{proof}
Let us check the conditions \ref{iv} and \ref{v} of Theorem~\ref{th-cond-1}. Under our assumptions, we have for any $n \geq 1$
\begin{gather*}
    Z_n (t)
    =\frac{B_{nt}}{\sqrt{n}}
    = W_n (t),
\end{gather*}
for some Wiener process $W_n$.

\medskip\noindent
Condition \ref{iv}: Note that the processes $Z_n$ and $\frac 1n\tau_{n\tinybullet}$ (``$\tinybullet$'' stands for the running argument $t$) are continuous. Furthermore, for any $0 \leq t_1 < t_2 <\ldots< t_k \leq T$, $k \geq 1$, and any $x_i, z_i \in \real$, $1 \leq i \leq k$, we have
\begin{align*}
    \Pp &\left\{\frac{\tau_{n t_i}}{n} \leq x_i, \; W_n (t_i) \leq z_i, \; \forall 1\leq i \leq k  \right\}\\
    &=\Pp \left\{\int_0^{nx_i} \frac{ds}{\lambda(B_s)} \geq nt_i, \; W_n (t_i) \leq z_i, \; \forall 1\leq i \leq k  \right\}\\
    &=\Pp \left\{\int_0^{x_i} \frac{ds}{\lambda(B_{ns})} \geq t_i, \; W_n (t_i) \leq z_i, \; \forall 1\leq i \leq k  \right\}\\
    &= \Pp \left\{\int_0^{x_i} \frac{ds}{\lambda\left(\sqrt{n} W_n (s)\right)} \geq t_i, \; W_n (t_i) \leq z_i, \; \forall 1\leq i \leq k  \right\}\\
    &= \Pp \left\{\int_0^{x_i} \frac{ds}{\lambda\left(\sqrt{n} W (s)\right)} \geq t_i, \; W (t_i) \leq z_i, \; \forall 1\leq i \leq k  \right\},
\intertext{where $W$ is a multi-dimensional  Wiener process. Furthermore,}
    \Pp&\left\{\int_0^{x_i} \frac{ds}{\lambda\left(\sqrt{n} W (s)\right)} \geq t_i, \; W (t_i) \leq z_i, \; \forall 1\leq i \leq k  \right\}\\&= \Pp \left\{\sum_{\alpha\in\{0,1\}^d} \int_0^{x_i} \frac{ \I_{\{W(s) \in \Delta_\alpha\}} }{\lambda\left(\sqrt{n} W(s)\right)}\, ds \geq t_i, \, W(t_{i}) \leq z_i, \, \forall 1 \leq i \leq k \right\}=: \Pp_n(k).
\end{align*}
Note that by Lebesgue's dominated convergence theorem
\begin{equation}\label{withp1}
	  \sum_{\alpha\in\{0,1\}^d}   \int_0^{x_i} \frac{ \I_{\{W_s \in \Delta_\alpha \}} }{\lambda\left(\sqrt{n} W_s\right)}\, ds
	\xrightarrow[n\to\infty]{}
	 \sum_{\alpha\in\{0,1\}^d}  \int_0^{x_i} \I_{\{W_s \in \Delta_\alpha \}} \,ds \,\frac{1}{\lambda_\alpha}
	  =: \nu(x_i). 
\end{equation}
Therefore, we have a weak convergence of probabilities
\begin{equation}\label{weakconv}
	\{\Pp_n(k)\}\xrightarrow[n\to\infty]{} \{\Pp \left\{    \nu(x_i) \geq t_i, \, W(t_{i}) \leq z_i, \, \forall 1 \leq i \leq k \right\} \},
\end{equation}
and in turn it means the weak convergence of finite-dimensional distributions of the respective processes. 
The trajectories of $\nu$ are strictly increasing in $x>0$ with probability $1$. Thus, the inverse process $\nu^{-1} (t)$ exists, and is a continuous process. In particular, we see that both limit processes are continuous. Finally,
\begin{align*}
    \Pp&\left\{ \frac{\tau_{n t_i}}{n} \leq x_i, \; W_n (t_i) \leq z_i, \;\forall 1\leq i \leq k  \right\}\\
    &\xrightarrow[n\to\infty]{} \Pp \big\{ \nu^{-1}(t_i) \leq x_i , \; W (t_i) \leq z_i, \;\forall 1\leq i \leq k  \big\},
\end{align*}
and the condition \ref{iv} is established.

\medskip\noindent
Condition \ref{v}: Obviously, we have for every $\epsilon>0$
\begin{align*}
    \lim_{h \downarrow 0} \limsup\limits_{n\to \infty} \Pp &\left\{ \sup_{0 \leq t_1 \leq t_2 \leq (t_1+h) \wedge T} \|Z_n (t_2) - Z_n (t_1)\|> \epsilon\right\} \\
    &=\lim_{h \downarrow 0} \Pp \left\{ \sup_{0 \leq t_1 \leq t_2 \leq (t_1+h)  \wedge T} \|W(t_2) - W (t_1)\|> \epsilon \right\}=0.
\end{align*}
 This proves condition \ref{v}, and an application of Theorem~\ref{th-cond-1} finishes the proof.
\end{proof}

\subsection{Functional limit theorems for the case of converging intensity $\lambda$}\label{sub3}

Our next aim is to remove the condition that the intensity $\lambda$ is separated from zero. From now on, we assume that $\lambda$ is  (Lebesgue a.e.) strictly positive with strictly positive limits in all octants and that $1/\lambda$  is locally integrable. There are two different approaches to the functional limit theorem for the time-changed multi-dimensional Wiener process, depending on the differing asymptotic behaviour of $\lambda$. The first approach  is considered in the next theorem, and it describes the situation, where the   limits at infinity of   $\lambda$ are ``radial''   (in each octant). This is expressed by the following assumption: 

\begin{enumerate}\setcounter{enumi}{5}
\item\label{vi} There exist $2^d$ numbers $\left\{ \lambda_\alpha,\: \alpha\in\{0,1\}^d \right\}$ such that
\begin{equation*}
	\lim_{\substack{\|x\| \to \infty\\x \in \Delta_\alpha}} \lambda(x) = \lambda_\alpha   > 0, 
	\quad\alpha\in\{0,1\}^d.
\end{equation*}
\end{enumerate}
Clearly, the assumption \ref{vi} implies assumption \ref{ii} from Lemma~\ref{infin}. In what follows, we use the standard notation $\cball(0,R)=\left\{x\in \real^d:\: \|x\|\leq R\right\}$;  $C$ denotes a constant whose value is not important and may change from line to line in a calculation.
\begin{theorem}\label{theor1-1}
	Let $\lambda: \real^d \to (0, \infty)$ be a measurable function, satisfying assumption \ref{i} from Lemma~\ref{infin} and assumption \ref{vi}. Additionally, we require
\begin{enumerate}\setcounter{enumi}{6}
	\item\label{vii}
	If $d=2$, then there exists some $\gamma\in(0,1/2)$ that for any $R>0$
	\begin{gather*}
		\int_{\cball(0, R)} \lambda^{-1} (x) \| x\|^{-2\gamma} \,dx  < \infty.
	\end{gather*}
	If  $d>2$, then for any $R>0$
	\begin{gather*}
		\int_{\cball(0, R)} \lambda^{-1} (x) \| x\|^{2-d} dx  < \infty.
	\end{gather*}
\end{enumerate}
Then on any interval $[0, T]$
\begin{equation*}
	\left\{ \frac{B_{\tau_{nt}}}{\sqrt{n}},\: t \in [0, T] \right\}
	\xRightarrow[n\to\infty]{} \left\{ W(\nu^{-1} (t)),\: t \in [0, T] \right\}.
\end{equation*}
\end{theorem}
\begin{remark}
	If $\lambda$ is locally integrable (assumption \ref{i}), then it is sufficient to assume \ref{vii} for any sufficiently small ball $\cball(0,\delta)$. In particular, it is enough to assume the integrability at the origin of $\lambda^{-1} (\cdot) \| \cdot\|^{-2\gamma}$ and $\lambda^{-1} (\cdot) \| \cdot\|^{2-d}$, for $d=2$ and $d>2$, respectively.
\end{remark}
\begin{proof}
Consider for any $t> 0 $ and $\alpha\in\{0,1\}^d$
\begin{equation}\label{ennoe}\begin{aligned}
    \Ee&\left|\int_0^t \frac{\I_{\{W_s \in \Delta_\alpha\}}\,ds}{\lambda\left(\sqrt{n} W_s^1,\dots ,\sqrt{n} W_s^d\right)} - \frac{1}{\lambda_\alpha} \int_0^t \I_{\{ W_s \in \Delta_\alpha\}}\, ds \right|\\
    &\leq \int_0^t \int_{\Delta_\alpha} \left|\frac{1}{\lambda\left(\sqrt{n} x\right)} - \frac{1}{\lambda_\alpha}\right| e^{-\frac{\|x\|^2}{2s}} \frac{1}{(2 \pi s)^{d/2}} \,dx \,ds
    =: \Ee_n(\Delta_\alpha).
\end{aligned}\end{equation}
Since all octants can be dealt with in a similar fashion, we can restrict ourselves to $\Delta_0 = \real_{+}^{d}$. The following change of variables
\begin{gather*}
		\frac{\|x\|^2}{2ns} = z,\quad
		s = \frac{\|x\|^2}{2nz},\quad
		ds = - \frac{\|x\|^2}{2nz^2}\,dz,
\end{gather*}
turns the formula \eqref{ennoe} for the first octant $\Delta_0=\real_+^d$ into
\begin{gather}\label{better}
	\Ee_n(\real_{+}^{d})
	= \frac{1}{2\pi^{d/2} n} \int_{\real_{+}^{d}} \int_{\frac{\|x\|^2}{2nt}}^{\infty} e^{-z} z ^{d/2 -2} \,dz  \left|\frac{1}{\lambda(x)}-\frac{1}{\lambda_\alpha}\right| \|x\|^{2-d} \,dx.
\end{gather}
We will now consider two cases.

\medskip\noindent
\textbf{Case 1:} $d=2$. We have
\begin{equation*}
    \Ee_n(\real_{+}^{2})
    = \frac{1}{2 n \pi} \int_{\real_{+}^{2}}
    \int_{\frac{\|x\|^2}{2nt}}^{\infty} e^{-z} z ^{-1}\,dz
    \left|\frac{1}{\lambda(x)}-\frac{1}{\lambda_0}\right| dx.
\end{equation*}
According to assumption \ref{vi}, for any fixed $\epsilon > 0$ there exists some $R = R_\epsilon >0$ such that $\left| \frac{1}{\lambda(x)} - \frac{1}{\lambda_0} \right| < \epsilon$ for all $x\in \Delta_1$ with $\|x\|\geq R$. Thus,
\begin{align*}
    \Ee_n(\real_{+}^{2})
    &\leq \frac{1}{2n \pi} \int_{\cball(0,R)\cap\real_{+}^{2}} \left(\frac{2nt}{\|x\|^2}\right)^{\gamma}
    \int_{\frac{\|x\|^2}{2nt}}^{\infty} e^{-z}z^{-1+\gamma} \, dz  \, \frac{1}{\lambda(x)}\,dx \\
    &\qquad\mbox{} + \frac{1}{\lambda_0} \frac{1}{2 n \pi} \int_{\cball(0,R)\cap\real_{+}^{2}} \left( \frac{2nt}{\|x\|^2}\right)^{\gamma}
    \int_{\frac{\|x\|^2}{2nt}}^{\infty}  e^{-z}z^{-1+\gamma} \,dz \,dx \\
    &\qquad\mbox{} + \frac{\epsilon}{2n\pi}  \int_{(\cball(0,R))^c\cap\real_{+}^{2}}
    \int_{\frac{\|x\|^2}{2nt}}^{\infty}  e^{-z}z^{-1} \,dz \,dx  \\
    &=: \mathrm{I}_1(n) + \mathrm{I}_2(n)+\mathrm{I}_3(n).
\end{align*}
From \ref{vii} we get
\begin{equation}\label{ione}
	\mathrm{I}_1(n)
	\leq \frac{t^{\gamma}\Gamma(\gamma)}{\pi (2n)^{1-\gamma}} \int_{\|x\|\leq R} \frac{1}{\lambda(x)} \| x\|^{-2\gamma} \,dx
	\xrightarrow[n\to\infty]{}0.
\end{equation}
Introducing polar coordinates in $\real^{2}$, we see that
\begin{equation}\label{idvan}
	\mathrm{I}_2(n)
	\leq \frac{1}{\lambda_0} \frac{t^{\gamma}\Gamma(\gamma)}{\pi (2n)^{1- \gamma}} \int_{\|x\| \leq R} \|x\|^{-2\gamma} \,dx
	= \frac{C}{\lambda_0} \frac{t^{\gamma}\Gamma(\gamma)}{\pi (2n)^{1- \gamma}} \int_0^R \frac{1}{r^{2\gamma}} r\,dr
	\xrightarrow[n\to\infty]{} 0.
\end{equation}
The third expression, $\mathrm{I}_3(n)$, can be bounded as follows:
\begin{equation}\label{itrin}
    \begin{aligned}
    \mathrm{I}_3(n)
    &\leq \frac{\epsilon C}{2n\pi} \int_{R}^{\infty} r \int_{\frac{r^2}{2nt}}^{\infty} e^{-z} z^{-1} \,dz \,dr\\
    &\leq \frac{\epsilon C}{2n\pi} \int_{R}^{\infty} r e^{-\frac{r^2}{4nt}} \frac{\sqrt{2nt}}{r} \int_0^{\infty} e^{-\frac{z}{2}} z^{-1/2} \,dz \,dr\\
    &\leq \frac{\epsilon C \sqrt{2t}}{2\pi \sqrt{n}} \int_{R}^{\infty} e^{-\frac{r^2}{4nt}} \,dr \\
    &\stackrel{\#}{\leq} \frac{\epsilon C \sqrt{2t}}{2\pi \sqrt{n}} \int_0^{\infty} e^{-z} \frac{\sqrt{nt}}{\sqrt{z}} \,dz
    \leq \frac{\epsilon Ct}{\pi\sqrt{2}}\Gamma\left(\tfrac{1}{2}\right)
    = \frac{\epsilon Ct}{\sqrt{2}}.
\end{aligned}
\end{equation}
In the step marked with ``$\#$'' we use the following change of variables: $z=\frac{r^2}{4nt}$, $r=  \sqrt{z}\cdot 2\sqrt{nt}$ and $dr =   \frac{dz}{\sqrt{z}} \cdot \sqrt{nt}$.

\medskip\noindent\textbf{Case 2:} $d>2$.
As in Case 1, for any $\epsilon > 0 $ there exists some $R = R_\epsilon > 0$ such that
\begin{gather*}
	\left|\frac{1}{\lambda(x)} - \frac{1}{ \lambda_0}\right| < \epsilon\quad\text{for all\ } \|x\|>R.
\end{gather*}
In contrast to dimension $d=2$ we have $\int_0^{\infty} e^{-z} z^{d/2 -2}\,dz < \infty$ if $d \geq 3$. Therefore,
\begin{equation}\label{longest}
\begin{aligned}
    \Ee _n(\real_{+}^{d})
    &\leq \frac{1}{2n\pi^{d/2}} \int_{\|x\|\leq R} \int_{\frac{\|x\|^2}{2nt}}^{\infty} e^{-z} z^{d/2-2} \,dz \,\frac{\|x\|^{2-d}}{\lambda(x)}\,dx\\
    &\qquad\mbox{} + \frac{1}{\lambda_0 2n\pi^{d/2}} \int_{\|x\| \leq R} \int_{\frac{\|x\|^2}{2nt}}^{\infty} e^{-z} z^{d/2-2}   \,dz  \,\frac{\|x\|^{2-d}}{\lambda(x)} \, dx\\
    &\qquad\mbox{} + \epsilon \frac{1}{2n\pi^{d/2}} \int_{\|x\| > R} \int_{\frac{\|x\|^2}{2nt}}^{\infty} e^{-z} z^{d/2-2} \,dz \,\|x\|^{2-d} \,dx\\
    &=: \mathrm{J}_1(n)+\mathrm{J}_2(n)+\mathrm{J}_3(n).
\end{aligned}
\end{equation}
Clearly,
\begin{equation}\label{longest1}\begin{aligned}
	\mathrm{J}_1(n) &\leq \frac{\Gamma(d/2-1)}{2n \pi^{d/2}} \int_{\|x\| \leq R} \frac{\|x\|^{2-d}}{\lambda(x)} \,dx,\\ \mathrm{J}_2(n) &\leq \frac{\Gamma(d/2-1)}{2\lambda_0n \pi^{d/2}} \int_{\|x\| \leq R} \|x\|^{2-d} \,dx,\\
	\mathrm{J}_3(n) &\leq \frac{C \epsilon}{2n \pi^{d/2}} \int_R^{\infty} \int_{\frac{r^2}{2nt}}^{\infty} e^{-z} z^{d/2-2} \,dz \cdot r^{2-d} \cdot r^{d-1} \,dr .
\end{aligned}\end{equation}
It is not hard to see that  $\mathrm{J}_1(n)$ and $\mathrm{J}_2(n)$ vanish as $n\to \infty$. The third term $\mathrm{J}_3(n)$ has the following upper bound:
\begin{equation}\label{jitrin}\begin{aligned}
	\mathrm{J}_3(n)
	&\leq \frac{C \epsilon}{2n \pi^{d/2}} \int_R^{\infty} \int_{\frac{r^2}{2nt}}^{\infty} e^{-z} z^{d/2-2} \,dz \cdot r^{2-d} \cdot r^{d-1} \,dr\\
	&\leq \frac{C \epsilon}{2n \pi^{d/2}} \int_{R}^{\infty} r \int_{\frac{r^2}{2nt}}^{\infty} e^{-z}z^{d/2-2} \,dz \,dr \\
    &\leq \frac{C \epsilon}{2n \pi^{d/2}} \int_{R}^{\infty} e^{-\frac{r^2}{4nt}} r \,dr \int_0^\infty e^{-z/2} z^{d/2-2} \,dz\\
    &\leq C\epsilon \int_0^{\infty} e^{-z} dz
    \leq C  \epsilon.
\end{aligned}\end{equation}
Combining all calculations \eqref{better}--\eqref{jitrin},   and letting first $n\to\infty$ and then $\epsilon\to 0$,   we conclude that for any $d\geq 2$ and $x>0$
\begin{gather*}
	\Ee \left|\int_0^{x}  \frac{1 }{\lambda\left(\sqrt{n} W_s\right)} \,ds - \int_0^{x}\sum_{\alpha\in\{0,1\}^d} \I_{\{ W_s \in \Delta_\alpha \}}  \,ds \,\frac{1}{\lambda_\alpha}\right|
	\xrightarrow[n\to\infty]{}0,
\end{gather*}
whence, in probability,
\begin{gather*}
	\int_0^{x} \frac{1}{\lambda\left(\sqrt{n} W_s\right)} \,ds
	\xrightarrow[n\to\infty]{} \int_0^{x}\sum_{\alpha\in\{0,1\}^d} \I_{\{ W_s \in \Delta_\alpha\}} \,ds \,\frac{1}{\lambda_\alpha}
\end{gather*}
Note that under conditions of Theorem~\ref{separated} we had almost sure convergence, see \eqref{withp1}, but convergence in probability is sufficient to  guarantee  weak convergence \eqref{weakconv} of the finite-dimensional distributions.  From now on we can literally follow that part of the proof of Theorem~\ref{separated}, which establishes assumption \ref{v} from Theorem~\ref{th-cond-1}.
\end{proof}

We will now consider the second possibility, i.e.\ the situation when the asymptotic behaviour of $\lambda$ is not ``radial'' (in each octant), but $\lambda$ has   ``diagonal''   limits in each octant. This behaviour is expressed by assumption \ref{ii}, and it is a weaker assumption than \ref{vi}. Therefore, some additional   property   of the asymptotic behaviour of $\lambda$ apart from the   ``diagonals''   is needed, and it will be given by assumption \ref{iii}.
\begin{theorem}\label{lim-thm-2}
	Let $\lambda: \real^d \to (0, \infty)$ be a measurable function, satisfying the assumptions \ref{i}, \ref{ii}, \ref{iii} and \ref{vii}. Then we have on any interval $[0, T]$
	\begin{equation*}
		\left\{ \frac{B_{\tau_{nt}}}{\sqrt{n}} , t \in [0, T] \right\}
		\xRightarrow[n\to\infty]{} \left\{ W(\nu^{-1} (t)), t \in [0, T] \right\}.
	\end{equation*}
\end{theorem}
\begin{proof}
Analysing the proofs of Theorems~\ref{separated} and~\ref{theor1-1} we see that it is enough to check that for any $\alpha\in\{0,1\}^d$
\begin{gather*}
	\Ee_n(\Delta_\alpha)\xrightarrow[n\to\infty]{} 0,
\end{gather*}
while all other necessary steps just repeat the proofs of these two theorems.

Since all octants can be treated in a similar way, we consider only $\Delta_0=\real_+^d$ and $\Ee_n(\real_{+}^{d})$ from \eqref{better}.   In view of the assumptions \ref{ii} and \ref{iii}, we pick some number $R>0$ and divide $\real_{+}^{d}$ into $2^d$ disjoint pieces: $\real_{+}^{d} = \bigcup_{m=1}^{2^d} \Theta_m$, where $\Theta_1=(0,R]^d$, $\Theta_2=(R,\infty)^d$, and $\Theta_3,\dots,\Theta_{2^d}$, which are Cartesian products each comprising $k$ factors of the form $(0,R]$ and $d-k$ factors of the form $(R,\infty)$, where $1\leq k<d$.  

The expression $\Ee_n(\Delta_0)$ is the sum of integrals $\int_{\Theta_1}+\int_{\Theta_2}+\dots+\int_{\Theta_{2^d}}$. Note that $\Theta_1\subset \cball(0,R)\cap\real_{+}^{d}$. Therefore, the first term $\int_{\Theta_1}$ can be bounded like $\mathrm{I}_1(n)$ and $\mathrm{I}_2(n)$ (see \eqref{ione} and \eqref{idvan}) in dimension $d=2$, and like $\mathrm{J}_1(n)$ and $\mathrm{J}_2(n)$ (see \eqref{longest} and \eqref{longest1}) in higher dimensions $d>2$.

Since $\Theta_2\subset\cball(0,R))^c\cap\real_{+}^{d}$, the expression $\int_{\Theta_2}$ can be bounded like $\mathrm{I}_3(n)$ (see \eqref{itrin}) and $\mathrm{J}_3(n)$ (see \eqref{jitrin}), in dimension $d=2$ and $d>2$, respectively.

It remains to bound $\int_{\Theta_m}$, $m=3,4,\dots, 2^d$. Without loss of generality we consider only the set $(0,R]^k\times (R,\infty)^{d-k}$, $1\leq k < d$, all other sets are dealt with in the same way. Write $\|\cdot\|_{n}$ for the Euclidean norm in $\real^{n}$, so that $\|x\|^2=\|x\|^2_{d-k}+\|x\|^2_{k}$.   Because of \ref{ii} we can pick some   $R>0$ such that $1/\lambda\leq C$ for some $C>0$ and all $x\in (0, R]^k \times (R, \infty)^{d-k}$. We have
\begin{align*}
    \frac{1}{n} &\int_{(0, R]^k \times (R, \infty)^{d-k}} \int_{\frac{\|x\|^2 }{2nt}}^{\infty} e^{-z} z^{d/2-2} \,dz
    \,\|x\|^{2-d}\,dx\\
    &\leq \frac{1}{n}  \int_{(0, R]^k \times (R, \infty)^{d-k}} \int_{\frac{\|x\|^2_{d-k} }{2nt}}^{\infty} e^{-z} z^{d/2-2} \,dz
    \,\|x\|_{d-k}^{2-d} \,dx\\
    &\leq \frac{R^k}{n}\int_{\|x\|_{d-k}\geq R}\int_{\frac{\|x\|^2_{d-k} }{2nt}}^{\infty} e^{-z} z^{d/2-2} \,dz
    \,\|x\|_{d-k}^{2-d} \,dx\\
    &\leq \frac{CR^k}{n}  \int_{R}^{\infty} \int_{\frac{r^2}{2nt}}^{\infty} e^{-z} z^{d/2-2} \,dz
    \,r^{d-k-1+2-d} \,dr.
\end{align*}
Choose $\gamma\in(1/2, 1)$ and observe that for every   $k=1,2,\dots,d-1$  
\begin{align*}
     \frac{CR^k}{n} &\int_{R}^{\infty} \int_{\frac{r^2}{2nt}}^{\infty} e^{-z} z^{d/2-2} \,dz   \,r^{d-k-1+2-d} \,dr\\
     &= \frac{CR^k}{n}  \int_{R}^{\infty}  r^{1-k}\int_{\frac{r^2}{2nt}}^{\infty} e^{-z} z^{-1} \,dz  \,dr\\
     &\leq \frac{CR^k(2nt)^\gamma}{n} \int_{R}^{\infty} r^{1-k-2\gamma}\int_{\frac{r^2}{2nt}}^{\infty} e^{-z} z^{d/2-2+\gamma} \,dz  \,dr\\
     &\leq C n^{\gamma-1} \int_{R}^{\infty}   r^{1-k-2\gamma} \,dr
     \leq Cn^{\gamma-1}\xrightarrow[n\to\infty]{} 0.
\qedhere
\end{align*}
\end{proof}

\begin{remark}\label{rem-gen}
Let us finally calculate the infinitesimal generator of the Markov process $Y^{(n)} = \left\{\frac{1}{\sqrt{n}}B_{\tau_{nt}},\,t\geq 0\right\}$, using the martingale representation \eqref{mart-repres}.  The sequence of processes
\begin{gather*}
	Y^{(n)}_t = \frac{B_{\tau_{nt}}}{\sqrt{n}} = \frac{Y_{nt}}{\sqrt{n}}
\end{gather*}
satisfies the equation
\begin{align*}
	Y^{(n)}_t
    = \frac 1{\sqrt n} \int_0^{nt}\lambda^{1/2}(Y_s)\,d\tilde{B}_s
    &= \frac 1{\sqrt n} \int_0^t\lambda^{1/2}(Y_{nz})\,d\tilde{B}_{nz}\\
    &= \int_0^t\lambda^{1/2}\left(\sqrt nY^{(n)}_{z}\right)\,d\tilde{B}^{(n)}_{z},
\end{align*}
where $\tilde{B}^{(n)}_{z} = n^{-1/2}\tilde{B}_{nz}$ is a $d$-dimensional Wiener process. Therefore, the generator of $Y^{(n)}$ is $A^{(n)}f(x) = \lambda\left(x\sqrt n\right) \Delta f(x)$; if $\lambda$ has limits in all octants, weak convergence of the processes $Y^{(n)}$ corresponds to the pointwise convergence of their generators. Let us describe the situation more precisely.
Denote by
\begin{gather*}
	\partial \real^{d}
	= \left\{x=(x_1,\ldots, x_d) \in \real^{d}: \text{at least one coordinate\ } x_i =0 \right\}.
\end{gather*}
Consider any coordinate of $Y^{(n)}_t = \left(Y^{(n);1}_t, \dots, Y^{(n);d}_t\right)$:
\begin{equation}\nonumber
	Y^{(n);i}_t
	= \int_0^t \lambda^{1/2} \left(\sqrt{n} Y^{(n)}_z\right) \,d\tilde{B}^{(n);i}_{z}.
\end{equation}
This process, being a local continuous martingale, has a local time, which we denote by  $L^{(n);i}_x  (t)$. For any $1 \leq i, k \leq d$ and any $C>0$
\begin{align*}
	\int_0^t &\left( \lambda \left( \sqrt{n}  Y^{(n)}_z  \right) \wedge C \right) \I_{\left\{Y^{(n);k}_z  =0  \right\}}\,dz \\
    &=  \int_0^t \frac{\left( \lambda \left(\sqrt{n}  Y^{(n)}_z  \right) \wedge C \right)}{\lambda \left(\sqrt{n}  Y^{(n)}_z  \right) }  \I_{\left\{ Y^{(n);k}_z  =0  \right\}} \,d\langle Y^{(n);k}\rangle_z \\
    &\leq  \int_{\real}  \I_{\{ 0\} }(x)  L^{(n);  k }_x  (t) \,dx
    = 0.
\end{align*}
This shows that
\begin{equation*}
	\int_0^t \lambda \left( \sqrt{n}  Y^{(n)}_z  \right) \I_{\{ Y^{(n)}_z \in \partial  \real^{d}  \}} \,dz =0,
\end{equation*}
whence
\begin{equation*}
	\int_0^t  \sqrt{\lambda \left(\sqrt{n}  Y^{(n)}_z\right)} \,d\tilde{B}^{(n);i}_{z}
	= \int_0^t \sqrt{\lambda \left(\sqrt{n}  Y^{(n)}_z\right)} \I_{\{ Y^{(n)}_z   \notin   \partial\real^{d}\}}\, d\tilde{B}^{(n);i}_{z}.
\end{equation*}
Since   $\lambda (\sqrt{n} x) \I_{\{  x \notin \partial  \real^{d} \}} \to \sum_{\alpha\in\{0,1\}^d}   \lambda_\alpha  \I_{\{ x \in \Delta_\alpha \}}$ as $n\to\infty$, we see that the limit process satisfies the following SDE
\begin{equation}\label{limit-gen}\begin{aligned}
	Y_t
	&= \int_0^t \sum_{\alpha\in\{0,1\}^d}  \sqrt{\lambda_\alpha}   \I_{\{ Y_s \in \Delta_\alpha \}}
	\, d\tilde{B}_s\\
	&= \int_0^t \left(\sum_{\alpha\in\{0,1\}^d} \sqrt{\lambda_\alpha}   \I_{\{ Y_s \in \Delta_\alpha \}}
	+ \I_{\{ Y_s \in \partial  \real^{d}\}}\right) d\tilde{B}_s
\end{aligned}\end{equation}
with some $d$-dimensional  Wiener process $\tilde{B}$. If the coefficients $\lambda_\alpha $ are strictly positive, then the diffusion coefficient of this equation is bounded, measurable and uniformly bounded away from $0$. According to the well-known  Krylov theorem (see, e.g.\ \cite[Proposition~1.15]{cher}), the equation \eqref{limit-gen} has a weak solution which is unique in law.
\end{remark}

\end{document}